\documentclass[oneside,reqno,american]{amsart}
\usepackage[T1]{fontenc}
\usepackage[utf8]{inputenc}
\usepackage{xcolor}
\usepackage{babel}
\usepackage{prettyref}
\usepackage{mathrsfs}
\usepackage{mathtools}
\usepackage{amstext}
\usepackage{amsthm}
\usepackage{amssymb}
\usepackage[pdfusetitle,
 bookmarks=true,bookmarksnumbered=false,bookmarksopen=false,
 breaklinks=false,pdfborder={0 0 0},pdfborderstyle={},backref=false,colorlinks=false]
 {hyperref}
\hypersetup{
 colorlinks=true,citecolor=blue,linkcolor=blue,linktocpage=true}

\makeatletter
\numberwithin{equation}{section}
\numberwithin{figure}{section}

\newrefformat{cor}{Corollary~\ref{#1}}
\newrefformat{subsec}{Section~\ref{#1}}
\newrefformat{lem}{Lemma~\ref{#1}}
\newrefformat{thm}{Theorem~\ref{#1}}
\newrefformat{sec}{Section~\ref{#1}}
\newrefformat{chap}{Chapter~\ref{#1}}
\newrefformat{prop}{Proposition~\ref{#1}}
\newrefformat{exa}{Example~\ref{#1}}
\newrefformat{tab}{Table~\ref{#1}}
\newrefformat{rem}{Remark~\ref{#1}}
\newrefformat{def}{Definition~\ref{#1}}
\newrefformat{fig}{Figure~\ref{#1}}
\newrefformat{claim}{Claim~\ref{#1}}

\makeatother

\theoremstyle{plain}
\newtheorem{thm}{\protect\theoremname}
\theoremstyle{definition}
\newtheorem{defn}[thm]{\protect\definitionname}
\theoremstyle{plain}
\newtheorem{cor}[thm]{\protect\corollaryname}
\theoremstyle{remark}
\newtheorem{rem}[thm]{\protect\remarkname}
\theoremstyle{plain}
\newtheorem{prop}[thm]{\protect\propositionname}
\newrefformat{cor}{Corollary \ref{#1}}
\newrefformat{def}{Definition \ref{#1}}
\newrefformat{prop}{Proposition \ref{#1}}
\newrefformat{rem}{Remark \ref{#1}}
\newrefformat{sec}{Section~\ref{#1}}
\newrefformat{subsec}{Section~\ref{#1}}
\newrefformat{thm}{Theorem \ref{#1}}
\providecommand{\corollaryname}{Corollary}
\providecommand{\definitionname}{Definition}
\providecommand{\propositionname}{Proposition}
\providecommand{\remarkname}{Remark}
\providecommand{\theoremname}{Theorem}

\begin{document}
\title{Random Operator-Valued Kernels and Moment Dilations }
\begin{abstract}
This paper studies random operator-valued positive definite (p.d.)
kernels and their connection to moment dilations. A class of random
p.d. kernels is introduced in which the positivity requirement is
imposed only in expectation, extending the deterministic case. Basic
properties are established, including convexity, sampling stability,
and the relationship between mean kernels and deterministic ones.
A Gaussian factorization result is proved under a Hilbert-Schmidt/trace
condition, clarifying the difference between pathwise and mean-square
positive definiteness and linking the construction to radonification
in the non-trace-class setting. Moment dilation is then formulated
for random operators, extending classical dilation theory to preserve
equality of mixed moments in expectation. The resulting dilation triples
admit interpretation in terms of unitary dynamics, connecting statistical
moment structure with operator-theoretic realization.
\end{abstract}

\author{James Tian}
\address{Mathematical Reviews, 535 W. William St, Suite 210, Ann Arbor, MI
48103, USA}
\email{jft@ams.org}
\maketitle

\section{Introduction}\label{sec:1}

Positive definite kernels are fundamental objects in operator theory,
harmonic analysis, and stochastic processes, where they connect covariance
structures to the geometry of Hilbert spaces (see, e.g., \cite{MR51437}).
In the classical setting, an operator-valued kernel $K\colon X\times X\rightarrow\mathcal{L}\left(H\right)$
is positive definite (p.d.) if all of its finite Gram matrices are
positive semidefinite, that is, 
\[
\sum_{i,j}\left\langle a_{i},K\left(s_{i},s_{j}\right)a_{j}\right\rangle \geq0
\]
for all finite $\left\{ \left(s_{i},a_{i}\right)\right\} $ in $X\times H$.
Here, $X$ is an arbitrary set (nonempty) and $H$ is a Hilbert space.
Throughout, inner products are assumed to be linear in the second
variable. $\mathcal{L}\left(H\right)$ denotes the space of all bounded
linear operators in $H$. 

This property is equivalent to the existence of a Kolmogorov/GNS type
decomposition 
\[
K\left(s,t\right)=V^{*}_{s}V_{t}
\]
where $V\colon H\rightarrow H'$ is an operator into an enlarged Hilbert
space $H'$. This deterministic notion underlies a wide range of constructions,
including reproducing kernel Hilbert spaces (RKHS) \cite{MR2184959,MR2055811,MR1897152}
and dilation theory \cite{MR2760647}.

In fact, $H'$ can be chosen to be the RKHS of an associated scalar-valued
kernel $\tilde{K}\colon\left(X\times H\right)^{2}\rightarrow\mathbb{C}$,
given by 
\[
\tilde{K}\left(\left(s,a\right),\left(t,b\right)\right)\coloneqq\left\langle a,K\left(s,t\right)b\right\rangle 
\]
defined for all $s,t\in X$ and $a,b\in H$. This method, by associating
a scalar-valued kernel $\tilde{K}$ to any operator-valued $K\colon X\times X\rightarrow\mathcal{L}\left(H\right)$,
was introduced by G. Pedrick \cite{MR2938971}; see also the exposition
in \cite{MR4250453} and more recent developments in \cite{MR4900382}.
A key advantage of this approach is that it avoids taking quotients
and working with equivalence classes, as in the classical GNS construction.
The dilation space is simply $H'=RKHS(\tilde{K})$, a Hilbert space
of functions on the product set $X\times H$ with the natural reproducing
property inherited from $\tilde{K}$.

The present work develops a probabilistic extension in which kernels
are allowed to vary randomly over a probability space $\left(\Omega,\mathbb{P}\right)$,
and the positivity requirement is imposed only on the expectation
of the associated Gram forms. Such random p.d. kernels form a natural
enlargement of the deterministic class, accommodating random perturbations,
stochastic models, and empirical approximations. Their mean kernels
are automatically p.d. in the deterministic sense, and they admit
structural results mirroring those of the classical theory, but with
additional subtleties arising from measurability and integrability
constraints.

A central aspect of this analysis is the relationship between these
random kernels and Gaussian process factorizations. Under a Hilbert-Schmidt/trace
condition on the diagonal, the mean kernel can be realized as the
covariance of an $H$-valued Gaussian process, and the random kernel
becomes pathwise positive definite almost surely. When this trace
condition fails, one may still represent the mean kernel in terms
of cylindrical Gaussian processes, but the associated random kernels
exist only as positive definite sesquilinear forms without bounded-operator
realization, reflecting the phenomenon of non-radonification.

In the later part of the paper, attention turns to moment dilation
for random operators, an extension of classical dilation theory in
which equalities of mixed moments are preserved in expectation. The
probabilistic kernel results developed here provide the foundational
tools for such constructions, as mean kernels arising from random
processes naturally encode the statistical moment structure needed
for dilation.

\subsection*{Literature Overview}

For classical treatments and applications of RKHS in approximation
theory, regularization, and inverse problems, see \cite{MR2184959,MR2055811,MR1897152}.
Recent developments have extended RKHS techniques into areas such
as tropical geometry and kernel learning theory \cite{MR4751741,MR4896090,MR4847432,MR2426053},
as well as operator-theoretic contexts involving Hilbert modules and
contractive multipliers \cite{MR2785850,MR3626500}. Analytic function
theory and universality limits provide additional fertile ground for
the kernel methods \cite{MR4707572}.

Classical dilation theory, from the Sz.-Nagy-Foias framework \cite{MR2760647}
to modern noncommutative multivariable generalizations \cite{MR69403,MR1858778,MR4248036},
includes foundational results for sequences of noncommuting operators
\cite{MR972704,MR1168596,MR1182494,MR1681749}, Wold decompositions
and doubly commuting isometries \cite{MR3151275}, and Hilbert module
approaches \cite{MR2891735}. Generalizations to unitary dilations
of several contractions and decompositions of positive definite kernels
are found in \cite{MR1902903,MR1160416}, while structured dilation
models for tetrablock contractions and commuting tuples appear in
\cite{MR4534539,MR4200246}. These developments underscore the role
of positive definite kernels as a unifying theme across deterministic
and probabilistic dilation theory.

The references given here are only a small sample of the extensive
literature on RKHS and dilation theory. For a broader view, see the
works cited above and the references therein, which together provide
a more comprehensive picture of the depth and variety of results in
these areas.

\section{Random Positive Kernels}

In the mean‐square analysis of random operators, a natural object
to study is the family of operator‐valued positive definite kernels
that encode their correlation structure. This section recalls the
basic definitions in both the deterministic and random settings, as
these kernels will serve as the starting point for the moment dilations
developed in \prettyref{sec:3}. In the deterministic case, positive
definiteness of $K\colon X\times X\to\mathcal{L}(H)$ means that all
finite Gram matrices built from $K$ are positive semidefinite. In
the random case, the same condition is imposed only in expectation
over the probability space $(\Omega,\mathbb{P})$, allowing individual
realizations to fail pointwise positivity while the averaged structure
remains positive. This averaged notion captures exactly the second‐moment
data needed for the constructions in the next section.

Let $X$ be a nonempty set, and $H$ a Hilbert space. Recall that
$K\colon X\times X\rightarrow\mathcal{L}\left(H\right)$ is positive
definite (p.d.) if 
\begin{equation}
\sum^{N}_{i,j=1}\left\langle a_{i},K\left(s_{i},s_{j}\right)a_{j}\right\rangle \geq0\label{eq:b1}
\end{equation}
for all finite points $\left\{ \left(a_{i},s_{i}\right)\right\} ^{N}_{i=1}$
in $X\times H$.

Let $\left(\Omega,\mathbb{P}\right)$ be a probability space. A map
$k\colon\Omega\times X\times X\rightarrow\mathcal{L}\left(H\right)$
is called a random p.d. kernel if, for all $s,t\in X$, $k\left(\cdot,s,t\right)$
is strongly measurable on $\Omega$, and 
\begin{equation}
\mathbb{E}\left[\sum^{N}_{i,j=1}\left\langle a_{i},k\left(\cdot,s_{i},s_{j}\right)a_{j}\right\rangle \right]\geq0\label{eq:b2}
\end{equation}
for all finite $\left\{ \left(a_{i},s_{i}\right)\right\} ^{N}_{i=1}$
in $X\times H$.
\begin{defn}
Let $pd\left(X,H\right)$ be the set of all operator-valued p.d. kernels
as in \eqref{eq:b1}. Denote by $rpd\left(\Omega,X,H\right)$ the
set of all random kernels satisfying \eqref{eq:b2}. 
\end{defn}

\begin{thm}
\label{thm:rpd}Given $k\in rpd\left(\Omega,X,H\right)$, let 
\begin{equation}
K\left(s,t\right)\coloneqq\mathbb{E}\left[k\left(\cdot,s,t\right)\right]=\int_{\Omega}k\left(\omega,s,t\right)d\mathbb{P}\left(\omega\right).\label{eq:b3}
\end{equation}
Then $K\in pd\left(X,H\right)$.

Conversely, for any $K\in pd\left(X,H\right)$, if $K\left(s,s\right)$
is of trace class for all $s\in X$, then there exists an $H$-valued
Gaussian process $\left\{ W_{s}\right\} _{s\in X}$ in $L^{2}\left(\Omega,H\right)$,
such that 
\[
K\left(s,t\right)=\mathbb{E}\left[\left|W_{s}\left(\cdot\right)\left\rangle \right\langle W_{t}\left(\cdot\right)\right|\right]
\]
with $k\left(\omega,s,t\right):=\left|W_{s}\left(\omega\right)\left\rangle \right\langle W_{t}\left(\omega\right)\right|\in rpd\left(\Omega,X,H\right)$.
In particular, this $k$ is pathwise p.d. a.s.
\end{thm}

\begin{proof}
Given $k\in rpd\left(\Omega,X,H\right)$, it is immediate that the
mean kernel from \eqref{eq:b3} belongs to $pd\left(X,H\right)$,
since 
\begin{align*}
\sum^{N}_{i,j=1}\left\langle a_{i},K\left(s_{i},s_{j}\right)a_{j}\right\rangle  & =\sum^{N}_{i,j=1}\left\langle a_{i},\mathbb{E}\left[k\left(\cdot,s_{i},s_{j}\right)\right]a_{j}\right\rangle \\
 & =\mathbb{E}\left[\sum^{N}_{i,j=1}\left\langle a_{i},k\left(\cdot,s_{i},s_{j}\right)a_{j}\right\rangle \right]\geq0,
\end{align*}
valid for all $\left(a_{i}\right)^{N}_{i=1}$ in $H$, $\left(s_{i}\right)^{N}_{i=1}$
in $X$, and all $N\in\mathbb{N}$. 

Conversely, for $K\in pd\left(X,H\right)$ let $K\left(s,t\right)=V^{*}_{s}V_{t}$
be its Kolmogorov decomposition, where $V_{s}\colon H\rightarrow H'$
is an operator from $H$ into a dilation space $H'$. Choose an orthonormal
basis $\left\{ e_{i}\right\} $ in $H'$, and let 
\[
W_{s}=\sum_{i}\left(V^{*}_{s}e_{i}\right)Z_{i}.
\]
Here, $Z_{i}$ is a system of i.i.d. random variables realized on
a probability space $\left(\Omega,\mathbb{P}\right)$, with $Z_{i}\sim N\left(0,1\right)$,
standard Gaussian. Then the series converges in $L^{2}\left(\Omega,H\right)$,
since 
\begin{align*}
\mathbb{E}\left[\left\Vert W_{s}\right\Vert ^{2}\right] & =\sum_{i,j}\left\langle V^{*}_{s}e_{i},V^{*}_{s}e_{j}\right\rangle \mathbb{E}\left[Z_{i}Z_{j}\right]\\
 & =\sum_{i}\left\Vert V^{*}_{s}e_{i}\right\Vert ^{2}=\sum\left\langle e_{i},V_{s}V^{*}_{s}e_{i}\right\rangle \\
 & =tr\left(V_{s}V^{*}_{s}\right)=tr\left(V^{*}_{s}V_{s}\right)=tr\left(K\left(s,s\right)\right)<\infty,\quad s\in X.
\end{align*}
Thus, $\left\{ W_{s}\right\} _{s\in X}$ is a well defined $H$-valued
Gaussian process, with 
\[
\left\langle W_{s},a\right\rangle =\sum_{i}\left\langle V^{*}_{s}e_{i},a\right\rangle Z_{i},\quad s\in X,\:a\in H.
\]
Its covariance is indeed $K$, since for all $a,b\in H$, 
\begin{align*}
\mathbb{E}\left[\left\langle a,W_{s}\right\rangle \left\langle W_{t},b\right\rangle \right] & =\sum_{i,j}\left\langle a,V^{*}_{s}e_{i}\right\rangle \left\langle V^{*}_{t}e_{j},b\right\rangle \mathbb{E}\left[Z_{i}Z_{j}\right]\\
 & =\sum_{i}\left\langle a,V^{*}_{s}e_{i}\right\rangle \left\langle V^{*}_{t}e_{i},b\right\rangle \\
 & =\sum_{i}\left\langle V_{s}a,e_{i}\right\rangle \left\langle e_{i},V_{t}b\right\rangle \\
 & =\left\langle V_{s}a,V_{t}b\right\rangle =\left\langle a,V^{*}_{s}V_{t}b\right\rangle =\left\langle a,K\left(s,t\right)b\right\rangle .
\end{align*}
This implies that 
\[
K\left(s,t\right)=\mathbb{E}\left(\left|W_{s}\left\rangle \right\langle W_{t}\right|\right)=\int_{\Omega}\left|W_{s}\left\rangle \right\langle W_{t}\right|d\mathbb{P}.
\]
Define the random p.d. kernel by 
\[
k\left(\omega,s,t\right)\coloneqq\left|W_{s}\left(\omega\right)\left\rangle \right\langle W_{t}\left(\omega\right)\right|\in rpd\left(\Omega,X,H\right).
\]
Clearly, $k\left(\omega,s,t\right)$ is also pathwise p.d. a.s. 
\end{proof}
A key assumption in the converse direction of \prettyref{thm:rpd}
is that 
\begin{equation}
tr\left(K\left(s,s\right)\right)<\infty,\quad s\in X,\label{eq:b4}
\end{equation}
or equivalently, in the Kolmogorov factorization $K\left(s,t\right)=V^{*}_{s}V_{t}$,
the operator $V_{s}$ is Hilbert-Schmidt (HS) for all $s\in X$. This
condition is also necessary: 
\begin{cor}
Let $K\in pd\left(X,H\right)$. There exists an $H$-valued Gaussian
process $\left\{ W_{s}\right\} _{s\in X}$ realized on some probability
space $\left(\Omega,\mathbb{P}\right)$ with 
\[
K\left(s,t\right)=\mathbb{E}\left[\left|W_{s}\left\rangle \right\langle W_{t}\right|\right],\quad s,t\in X
\]
if and only if 
\[
tr\left(K\left(s,s\right)\right)<\infty,\quad\forall s\in X.
\]
\end{cor}

\begin{proof}
One needs only to check that, if $K\left(s,t\right)=\mathbb{E}\left[\left|W_{s}\left\rangle \right\langle W_{t}\right|\right]$,
then 
\begin{align*}
tr\left(K\left(s,s\right)\right) & =\sum\nolimits_{i}\left\langle e_{i},K\left(s,s\right)e_{i}\right\rangle \\
 & =\mathbb{E}\left[\sum\nolimits_{i}\left|\left\langle W_{s},e_{i}\right\rangle \right|^{2}\right]=\mathbb{E}\left[\left\Vert W_{s}\right\Vert ^{2}\right]<\infty
\end{align*}
since $\left\{ W_{s}\right\} _{s\in X}\subset L^{2}\left(\Omega,H\right)$
by definition. Here, $\left\{ e_{i}\right\} $ is any ONB in $H$. 
\end{proof}
\begin{rem}
Consequently, if the trace/HS condition \eqref{eq:b4} is dropped,
there is no $H$-valued Gaussian process $\left\{ W_{s}\right\} _{s\in X}$
with covariance $K\left(s,t\right)$.

Nevertheless, fix an isonormal Gaussian process $G$ on $H'$. For
each $s\in X$ define the cylindrical Gaussian 
\[
W^{\circ}_{s}\left(a\right)=G\left(V_{s}a\right),\quad a\in H.
\]
Then for all $a,b\in H$, 
\[
\mathbb{E}\left[\overline{W^{\circ}_{s}\left(a\right)}W^{\circ}_{t}\left(b\right)\right]=\left\langle V_{s}a,V_{t}b\right\rangle =\left\langle a,K\left(s,t\right)b\right\rangle ,
\]
and for any finite $\left\{ \left(s_{i},a_{i}\right)\right\} ^{N}_{i=1}$,
\[
\sum^{N}_{i,j=1}\overline{W^{\circ}_{s_{i}}\left(a_{i}\right)}W^{\circ}_{s_{j}}\left(a_{j}\right)=\left|\sum^{N}_{i=1}G\left(V_{s_{i}}a\right)\right|^{2}\geq0\quad a.s.
\]
Thus the random kernel defined as a sesquilinear form 
\[
\left\langle a,k\left(\omega,s,t\right)b\right\rangle \coloneqq\overline{W^{\circ}_{s}\left(a\right)\left(\omega\right)}W^{\circ}_{s}\left(b\right)\left(\omega\right)
\]
is pathwise positive definite and satisfies $\mathbb{E}\left[k\left(\cdot,s,t\right)\right]=K\left(s,t\right)$.
In general, however, $k\left(\omega,s,t\right)$ need not correspond
to a bounded operator on $H$; it is merely a random positive definite
form. This is precisely the failure of radonification (see e.g., \cite{MR1930962,MR4595378})
in the non-trace-class case.
\end{rem}

\begin{rem}
A useful approximation scheme is obtained by finite-rank truncation.
Let $\left\{ e_{i}\right\} _{i\in I}$ be an orthonormal basis (possibly
uncountable) of $H'$. For each finite $F\subset I$, let $P_{F}=\sum_{i\in F}\left|e_{i}\left\rangle \right\langle e_{i}\right|$
be the projection onto $span\left\{ e_{i}:i\in F\right\} $, and define
\[
W^{\left(F\right)}_{s}\coloneqq\sum_{i\in F}\left(V^{*}_{s}e_{i}\right)Z_{i}\in L^{2}\left(\Omega,H\right),\quad k_{F}\left(\omega,s,t\right)\coloneqq\left|W^{\left(F\right)}_{s}\left(\omega\right)\left\rangle \right\langle W^{\left(F\right)}_{t}\left(\omega\right)\right|.
\]
Then $W^{\left(F\right)}_{s}$ is $H$-valued Gaussian with 
\[
\mathbb{E}\left[\left\langle a,k_{F}\left(\cdot,s,t\right)b\right\rangle \right]=\sum_{i\in F}\left\langle a,V^{*}_{s}e_{i}\right\rangle \left\langle V^{*}_{t}e_{i},b\right\rangle =\left\langle a,V^{*}_{s}P_{F}V_{t}b\right\rangle ,
\]
so 
\[
\mathbb{E}\left[k_{F}\left(\cdot,s,t\right)\right]=V^{*}_{s}P_{F}V_{t}\rightarrow K\left(s,t\right),\;\text{as }F\rightarrow I
\]
(directed by inclusion) in the weak operator topology. Moreover, 
\[
\sum_{i,j}\left\langle a_{i},k_{F}\left(\omega,s_{i},s_{j}\right)a_{j}\right\rangle =\left|\sum_{i}\left\langle W^{\left(F\right)}_{s}\left(\omega\right),a_{i}\right\rangle \right|^{2}\geq0,\;a.s.
\]
Note that 
\[
\mathbb{E}\left\Vert W^{\left(F\right)}_{s}\right\Vert ^{2}=tr\left(V^{*}_{s}P_{F}V_{s}\right)\rightarrow tr\left(V^{*}_{s}V_{s}\right)=tr\left(K\left(s,s\right)\right)=\infty,\quad F\rightarrow I.
\]
so $\{W^{\left(n\right)}_{s}\}_{F}$ (fixed $s$) has no $H$-valued
limit in $L^{2}\left(\Omega,H\right)$.
\end{rem}

\subsection{The class $rpd\left(\Omega,X,H\right)$}

Condition \eqref{eq:b2} is much weaker than pathwise positive definiteness
because it constrains only the average of the Gram forms, not each
realization: for a given finite set $\left\{ \left(s_{i},a_{i}\right)\right\} $,
the random matrix $\left[\left\langle a_{i},k\left(\cdot,s_{i},s_{j}\right)a_{j}\right\rangle \right]_{i,j}$
may be indefinite with positive probability (e.g. add a mean-zero
indefinite perturbation), but its expectation remains positive definite. 

Both $rpd\left(\Omega,X,H\right)$ and $pd\left(X,H\right)$ are convex
cones, and there is a natural containment 
\[
pd\left(X,H\right)\subset rpd\left(\Omega,X,H\right)
\]
by identifying any deterministic $K$ with the constant random kernel,
i.e., $k\left(\omega,s,t\right)\equiv K\left(s,t\right)$. In fact,
$rpd\left(\Omega,X,H\right)$ can be seen as a sampling stable enlargement
of $pd\left(X,H\right)$ in the sense that empirical averages of i.i.d.
samples from a random kernel remain in the class and converge to the
mean kernel. This is made precise by the following proposition.
\begin{prop}
\label{prop:6}Take $k\in rpd\left(\Omega,X,H\right)$. Let $k^{\left(1\right)},\dots,k^{\left(m\right)}$
be i.i.d. copies of $k$ defined on the product space $\left(\Omega^{m},\mathbb{P}^{\otimes m}\right)$.
Define the empirical kernel 
\[
\overline{k}_{m}\left(\omega_{1},\dots,\omega_{m},s,t\right)\coloneqq\frac{1}{m}\sum^{m}_{r=1}k^{\left(r\right)}\left(\omega_{r},s,t\right).
\]
Then $\overline{k}_{m}\in rpd\left(\Omega^{m},X,H\right)$, and $\mathbb{E}\overline{k}_{m}=K$.
Moreover, 
\[
\overline{k}_{m}\left(\cdot,s,t\right)\xrightarrow[a.s.]{\;m\rightarrow\infty\;}K\left(s,t\right)
\]
in the weak operator topology.
\end{prop}

\begin{proof}
For any finite data $\left\{ \left(s_{i},a_{i}\right)\right\} ^{N}_{i=1}$
in $X\times H$, 
\[
\mathbb{E}\left[\sum^{N}_{i,j=1}\left\langle a_{i},\overline{k}_{m}\left(\cdot,s_{i},s_{j}\right)a_{j}\right\rangle \right]=\frac{1}{m}\sum^{m}_{r=1}\mathbb{E}\left[\sum^{N}_{i,j=1}\left\langle a_{i},k^{\left(r\right)}\left(\cdot,s_{i},s_{j}\right)a_{j}\right\rangle \right]\geq0,
\]
since each $k^{\left(r\right)}$ is a copy of $k\in rpd\left(\Omega,X,H\right)$.
Thus $\overline{k}_{m}\in rpd\left(\Omega^{m},X,H\right)$, and $\mathbb{E}\left[\overline{k}_{m}\right]=K$
by i.i.d. and linearity.

On the other hand, for all $a,b\in H$, 
\[
\left\langle a,\overline{k}_{m}\left(\cdot,s,t\right)b\right\rangle =\frac{1}{m}\sum^{m}_{r=1}\left\langle a,k^{\left(r\right)}\left(\cdot,s,t\right)b\right\rangle \xrightarrow[a.s.]{\;m\rightarrow\infty\;}\left\langle a,K\left(s,t\right)b\right\rangle 
\]
by the strong law of large numbers. 
\end{proof}

\section{Moment dilation}\label{sec:3}

A central question in the mean‐square dilation framework is when the
sequence of mixed moments
\begin{equation}
K\left(m,n\right)\coloneqq\mathbb{E}\left[A^{*m}A^{n}\right],\quad m,n\in\mathbb{N}_{0}=\left\{ 0,1,2,\dots\right\} ,\label{eq:c1}
\end{equation}
arising from a random operator $A\colon\Omega\rightarrow\mathcal{L}\left(H\right)$,
can be represented as the compression of powers of a single unitary
operator on a larger Hilbert space. Such a representation preserves
the entire moment structure of $A$ in a “unitary model” and generalizes
the classical Sz.-Nagy power dilation from deterministic contractions
to the random, mean‐square setting. This section introduces the notion
of moment dilation, characterizes its existence, and relates them
to positivity certain Toeplitz‐type inequalities satisfied by $K\left(m,n\right)$.
\begin{defn}
\label{def:7}Let $\left(\Omega,\mathbb{\mathbb{P}}\right)$ be a
Borel probability space. Consider random operators $A\colon\Omega\rightarrow\mathcal{L}\left(H\right)$
that are strongly measurable. An operator $A$ is said to admit a
moment dilation if there exist an isometry $W\colon H\rightarrow\mathcal{K}$,
a unitary $U\in\mathcal{L}\left(\mathcal{K}\right)$, and a selfadjoint
projection $P\in\mathcal{L}\left(\mathcal{K}\right)$, such that its
moment kernel has a Kolmogorov/GNS type factorization of the form:
\begin{equation}
\mathbb{E}\left[A^{*m}A^{n}\right]=W^{*}U^{*m}PU^{n}W\label{eq:c2}
\end{equation}
for all $m,n\in\mathbb{N}_{0}$. In particular, 
\[
\mathbb{E}\left[A^{n}\right]=W^{*}PU^{n}W
\]
for all $n\in\mathbb{N}_{0}$.
\end{defn}

The projection $\ensuremath{P}$ accounts for the fact that, in the
random setting, the forward and backward moment sequences $\mathbb{E}[A^{*m}A^{n}]$
may not arise from the compression of a unitary acting on a single
embedded copy of $H$. Instead, $P$ specifies the subspace of the
dilation space $\mathcal{K}$ on which the compressed unitary action
reproduces the mean-square moments. In the deterministic case, where
$A$ is a contraction on $H$ with no randomness, the identity 
\[
A^{*m}A^{n}=W^{*}U^{*m}U^{n}W
\]
is recovered with $P=I_{\mathcal{K}}$, reducing to the classical
Sz.-Nagy power dilation theorem. Thus, $P$ can be seen as a ``filter''
through which the unitary dynamics are observed in order to match
the statistical second-moment structure of $A$.

In this formulation, the dilation is designed to reproduce the entire
matrix of mixed second moments
\[
\mathbb{E}\left[A^{*m}A^{n}\right],\quad m,n\in\mathbb{N}_{0},
\]
not just the first moment $\mathbb{E}\left[A\right]$. These quantities
determine, for each pair of integers $m,n$ the expected inner products
\[
\left\langle a,\mathbb{E}\left[A^{*m}A^{n}\right]b\right\rangle =\mathbb{E}\left[\left\langle A^{m}a,A^{n}b\right\rangle \right],\quad a,b\in H
\]
and hence encode the full second-order correlation structure of the
sequence $\left(A^{n}\right)_{n\geq0}$ in the mean square sense. 

This is statistical second-moment structure in the Hilbert space sense:
the data $\mathbb{E}\left[A^{*m}A^{n}\right]$ controls all quadratic
forms in the input vectors but does not determine the full law of
the random operator $A\left(\omega\right)$. In particular, higher-order
cumulants are not captured. Nonetheless, knowing the entire family
$\mathbb{E}\left[A^{*m}A^{n}\right]$ is strictly stronger than knowing
only $\mathbb{E}\left[A\right]$, and retains essential information
about fluctuations that would be lost if one tried to dilate $\mathbb{E}\left[A\right]$
alone. 

Below we characterize when such a moment dilation exists.
\begin{thm}
\label{thm:8}Let $A\colon\Omega\rightarrow\mathcal{L}\left(H\right)$
be a random operator with moment kernel $K$ as in \eqref{eq:c1},
and set 
\[
K_{shift}\left(m,n\right)\coloneqq K\left(m+1,n+1\right)
\]
for all $m,n\in\mathbb{N}_{0}$.

The following are equivalent:
\begin{enumerate}
\item \label{enu:8a}There exists a moment dilation $\left(\mathcal{K},U,P,W\right)$
(see \prettyref{def:7}) with 
\[
K\left(m,n\right)=W^{*}U^{*m}PU^{n}W,
\]
and 
\begin{equation}
U^{*}PU\leq P\;\text{on }\overline{span}\left\{ U^{k}WH:k\in\mathbb{N}_{0}\right\} .\label{eq:c3}
\end{equation}
\item \label{enu:8b}$K_{shift}\leq K$, i.e., $K-K_{shift}$ is p.d.
\end{enumerate}
In particular, if $K_{shift}\leq K$, the operator $A$ admits a moment
dilation in the sense of \prettyref{eq:c2} subject to the constraint
\eqref{eq:c3}.

\end{thm}

\begin{proof}
The argument follows the classical dilation machinery, adapted to
the mean-square setting.

First, let $K\colon\mathbb{N}_{0}\times\mathbb{N}_{0}\rightarrow\mathbb{C}$
be the $\mathcal{L}\left(H\right)$-valued moment kernel from \eqref{eq:c1},
i.e., 
\[
K\left(m,n\right)\coloneqq\mathbb{E}\left[A^{*m}A^{n}\right].
\]
For all finite $\left\{ \left(m_{i},a_{i}\right)\right\} ^{N}_{i=1}$
in $\mathbb{N}_{0}\times H$, 
\begin{align*}
\sum\nolimits^{N}_{i,j=1}\left\langle a_{i},K\left(m_{i},m_{j}\right)a_{j}\right\rangle  & =\mathbb{E}\left[\sum\nolimits^{N}_{i,j=1}\left\langle A^{m_{i}}a_{i},A^{m_{j}}a_{j}\right\rangle \right]\\
 & =\mathbb{E}\left[\left\Vert \sum\nolimits^{N}_{i=1}A^{m_{i}}a_{i}\right\Vert ^{2}\right]\geq0,
\end{align*}
so $K$ is positive definite. 

\eqref{enu:8a}$\Longrightarrow$\eqref{enu:8b} Suppose
\[
K\left(m,n\right)=\mathbb{E}\left[A^{*m}A^{n}\right]=W^{*}U^{*m}PU^{n}W
\]
satisfying \eqref{eq:c3}. For any finite $\left\{ \left(m_{i},a_{i}\right)\right\} $
in $\mathbb{N}_{0}\times H$, 
\begin{align*}
\sum\nolimits_{i,j}\left\langle a_{i},K_{shift}\left(m_{i},m_{j}\right)a_{j}\right\rangle  & =\sum\nolimits_{i,j}\left\langle a_{i},K\left(m_{i}+1,m_{j}+1\right)a_{j}\right\rangle \\
 & =\sum\nolimits_{i,j}\left\langle U^{m_{i}+1}Wa_{i},PU^{m_{j}+1}Wa_{j}\right\rangle \\
 & =\left\Vert \sum\nolimits_{i}PU^{m_{i}+1}Wa_{i}\right\Vert ^{2}=\left\Vert PUy\right\Vert ^{2}
\end{align*}
where 
\[
y:=\sum\nolimits_{i}U^{m_{i}}Wa_{i}\in\overline{span}\left\{ U^{k}WH:k\in\mathbb{N}_{0}\right\} .
\]
Similarly, 
\[
\sum\nolimits_{i,j}\left\langle a_{i},K\left(m_{i},m_{j}\right)a_{j}\right\rangle =\left\Vert \sum\nolimits_{i}PU^{m_{i}}Wa_{i}\right\Vert ^{2}=\left\Vert Py\right\Vert ^{2}.
\]
By \eqref{eq:c3}, 
\[
\left\Vert PUy\right\Vert ^{2}=\left\langle y,U^{*}PUy\right\rangle \leq\left\langle y,Py\right\rangle =\left\Vert Py\right\Vert ^{2}.
\]
Therefore, $K_{shift}\leq K$.

\eqref{enu:8b}$\Longrightarrow$\eqref{enu:8a} Here, we construct
the data $\left(\mathcal{K},U,P,W\right)$, assuming $K_{shift}\leq K$. 

By the Kolmogorov/GNS factorization, 
\[
K\left(m,n\right)=V^{*}_{m}V_{n},\quad m,n\in\mathbb{N}_{0},
\]
where $V_{n}\colon H\rightarrow H'$ maps $H$ into a larger space
$H'$, and $span\left\{ V_{n}a:a\in H,n\in\mathbb{N}_{0}\right\} $
is dense in $H'$. (One may choose $H'$ to be the RKHS of the associated
scalar-valued kernel $\tilde{K}$, as discussed in \prettyref{sec:1}.) 

Define a shift operator $B$ on $H'$ by $BV_{n}a=V_{n+1}a$, and
extend linearly. Note that 
\begin{align*}
\left\Vert B\left(\sum\nolimits_{i}V_{i}a_{i}\right)\right\Vert ^{2} & =\left\Vert \sum\nolimits_{i}V_{i+1}a_{i}\right\Vert ^{2}\\
 & =\sum\nolimits_{i,j}\left\langle a_{i},V^{*}_{i+1}V_{j+1}a_{j}\right\rangle \\
 & =\sum\nolimits_{i,j}\left\langle a_{i},K_{shift}\left(i,j\right)a_{j}\right\rangle \\
 & \leq\sum\nolimits_{i,j}\left\langle a_{i},K\left(i,j\right)a_{j}\right\rangle =\left\Vert \sum\nolimits_{i}V_{i}a_{i}\right\Vert ^{2}
\end{align*}
by the assumption $K_{shift}\leq K$. Thus, $B$ extends to a contraction
on $H'$ by density, i.e., $\left\Vert B\right\Vert \leq1$. 

We then apply the Sz.-Nagy dilation theorem to $B$. (Pedrick's scalar-valued
kernel trick \cite{MR2938971} works here, too; see e.g., \cite[sect. 2.1]{MR4760560}.)
There exists an isometry $J\colon H'\rightarrow\mathcal{K}$, a unitary
$U$ on $\mathscr{K}$, such that 
\[
B^{n}=J^{*}U^{n}J,\quad n\in\mathbb{N}_{0}.
\]
Then, $K$ factors into 
\begin{align*}
K\left(m,n\right) & =V^{*}_{m}V_{n}=V^{*}_{0}B^{*m}B^{n}V_{0}\\
 & =V^{*}_{0}\left(J^{*}U^{*m}J\right)\left(J^{*}U^{n}J\right)V_{0}\\
 & =\left(V^{*}_{0}J^{*}\right)U^{*m}\left(JJ^{*}\right)U^{n}\left(JV_{0}\right)\\
 & =W^{*}U^{*m}PU^{n}W
\end{align*}
with $W=JV_{0}\colon H\rightarrow\mathcal{K}$ isometric, $P=JJ^{*}$. 

Moreover, using the argument in the first part of the proof, it is
clear that 
\[
K_{shift}\leq K\Longleftrightarrow U^{*}PU\leq P\:\text{on }\overline{span}\left\{ U^{k}WH\right\} \Longleftrightarrow\left\Vert B\right\Vert \leq1.
\]
This completes the proof.
\end{proof}
\begin{cor}
$P=I_{\mathcal{K}}$ if and only if $K_{shift}=K$. That is, the operator
$A$ admits a moment dilation of the form 
\[
K\left(m,n\right)=\mathbb{E}\left[A^{*m}A^{n}\right]=W^{*}U^{*m}U^{n}W
\]
if and only if $K$ is shift-invariant (stationary). 
\end{cor}

\begin{proof}
Follows from the proof of \prettyref{thm:8}. Details are omitted. 
\end{proof}
As an application, we obtain the following mean-square variant of
the classical von Neumann inequality. Recall that in the deterministic
setting, if $T$ is a contraction on a Hilbert space $H$, then 
\[
\left\Vert f\left(T\right)\right\Vert \le\sup_{\left|z\right|=1}\left|f\left(z\right)\right|
\]
for every complex polynomial $f$. A standard proof uses Sz.-Nagy's
unitary dilation, a well‐known approach appearing in numerous standard
texts (see, e.g., \cite{MR1858778,MR2743416,MR4248036}).

In the random setting, the argument follows the same outline, but
with two key modifications: 
\begin{enumerate}
\item The deterministic moments $T^{*m}T^{n}$ are replaced by the mean‐square
moment kernel $\mathbb{E}\left[T^{*m}T^{n}\right]$. 
\item The dilation now involves a projection $P$ in the factorization \eqref{eq:c2}. 
\end{enumerate}
With these, the spectral theorem for $U$ still controls the norm
of $f\left(U\right)$, leading directly to the inequality below.
\begin{prop}
Let $A\colon\Omega\rightarrow\mathcal{L}\left(H\right)$ be a random
operator. Suppose it admits a moment dilation, so that $\mathbb{E}\left[A^{*m}A^{n}\right]=W^{*}U^{*m}PU^{n}W$
as above. (This holds in particular when $\left\Vert A\left(\cdot\right)\right\Vert \leq1$
a.s.)

Then, for every complex polynomial $f\left(z\right)=\sum c_{n}z^{n}$
and $a\in H$, 
\[
\mathbb{E}\left[\left\Vert f\left(A\right)a\right\Vert ^{2}\right]\leq\left(\sup\nolimits_{\left|z\right|=1}\left|f\left(z\right)\right|\right)^{2}\left\Vert a\right\Vert ^{2}.
\]
Equivalently, 
\[
\mathbb{E}\left[f\left(A\right)^{*}f\left(A\right)\right]\leq\left(\sup\nolimits_{\left|z\right|=1}\left|f\left(z\right)\right|\right)^{2}I_{H}.
\]
\end{prop}

\begin{proof}
One checks that 
\begin{align*}
\mathbb{E}\left[\left\Vert f\left(A\right)a\right\Vert ^{2}\right] & =\mathbb{E}\left[\left\Vert \sum\nolimits_{m}c_{m}A^{m}a\right\Vert ^{2}\right]\\
 & =\sum\nolimits_{m,n}\overline{c_{m}}c_{n}\left\langle x,\mathbb{E}\left[A^{*m}A^{n}\right]a\right\rangle \\
 & =\sum\nolimits_{m,n}\overline{c_{m}}c_{n}\left\langle a,W^{*}U^{*m}PU^{n}Wa\right\rangle \\
 & =\left\Vert Pf\left(U\right)Wa\right\Vert ^{2}\\
 & \leq\left\Vert f\left(U\right)\right\Vert ^{2}\left\Vert Wa\right\Vert ^{2}\\
 & =\left(\sup\nolimits_{\left|z\right|=1}\left|f\left(z\right)\right|\right)^{2}\left\Vert a\right\Vert ^{2}.
\end{align*}
Here, we used that $W$ is an isometry, $\left\Vert Wa\right\Vert =\left\Vert a\right\Vert $,
and applied the spectral theorem to the unitary $U$. 
\end{proof}
\begin{cor}
If $A$ has a moment dilation, then
\[
\mathbb{E}\left[\left\Vert Aa\right\Vert ^{2}\right]\leq\left\Vert a\right\Vert ^{2},\quad a\in H.
\]
The converse is not true in general. 
\end{cor}

\bibliographystyle{amsalpha}
\bibliography{ref}

\end{document}